\documentclass[reqno]{amsart}
\usepackage[english]{babel}
\usepackage{graphicx,subfigure}
\usepackage{amsmath,amssymb}
\usepackage{color}
\definecolor{oneblue}{rgb}{0,0.0,0.75}

\usepackage{fancyhdr}

\makeatletter
\newcommand{\sech}{\mathop{\operator@font sech}}
\newcommand{\sign}{\mathop{\operator@font sign}}
\makeatother

\newtheorem{proposition}{Proposition}[section]

\numberwithin{equation}{section}

\begin{document}

\title[]{Numerical solution of internal-wave systems in the intermediate long wave and the Benjamin-Ono regimes}

\author[V. A. Dougalis]{Vassilios A. Dougalis}
\address{Mathematics Department, University of Athens, 15784
Zographou, Greece \and Institute of Applied \& Computational
Mathematics, FO.R.T.H., 71110 Heraklion, Greece}
\email{doug@math.uoa.gr}

\author[A. Duran]{Angel Duran}
\address{ Applied Mathematics Department,  University of
Valladolid, 47011 Valladolid, Spain}
\email{angel@mac.uva.es}

\author[L. Saridaki]{Leetha Saridaki}
\address{Mathematics Department, University of Athens, 15784
Zographou, Greece \and Institute of Applied \& Computational
Mathematics, FO.R.T.H., 71110 Heraklion, Greece}
\email{leetha.saridaki@gmail.com}


\begin{abstract}
The paper is concerned with the numerical approximation of the Intermediate Long Wave and Benjamin-Ono systems, that serve as models for the propagation of interfacial internal waves in a two-layer fluid system in particular physical regimes. The paper focuses on two issues of approximation. First, the spectral Fourier-Galerkin method is used to discretize in space the corresponding periodic initial-value problems, and the error of the semidiscretizations is analyzed. The second issue concerns the numerical generation of solitary-wave solutions of the systems. We use acceleration techniques to improve the computation of the approximate solitary waves and check their performance with numerical examples.
\end{abstract}

\maketitle

\section{Introduction}
In this paper we consider the numerical approximation of two one-dimensional, nonlocal systems of nonlinear partial differential equations (pde's)  of dispersive wave type. The systems have been derived in \cite{BLS2008} as models describing the propagation of internal waves in a two-layer interface problem with rigid upper and lower boundaries and under two different regimes in the case of a shallow upper layer and small-amplitude deformations in the lower layer.
\begin{figure}[htbp]
\centering
\includegraphics[width=0.8\textwidth]{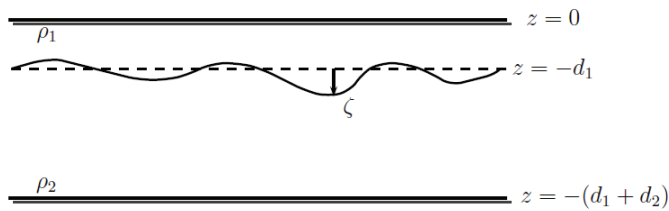}
\caption{Idealized model of internal wave propagation in a two-layer interface problem: $\rho_{2}>\rho_{1}; d_{2}>d_{1}$; $\zeta(x,t)$ denotes the downward vertical displacement of the interface from its level of rest at $(x,t)$.}
\label{ILWBO_fig1}
\end{figure}
The idealized model is sketched in Figure \ref{ILWBO_fig1}. It consists of two layers of inviscid, homogeneous, incompressible fluids with depths $d_{1}, d_{2}$ and densities $\rho_{1}<\rho_{2}$. The upper and lower layers are bounded above and below, respectively, by rigid lids, with the origin of the vertical coordinate at the top.

In \cite{BLS2008} the Euler equations with interface are reformulated in terms of two nonlocal operators linking the velocity potentials associated with the two layers and evaluated at the interface. This approach is then used to derive different asymptotic models, with which the Euler system is consistent. The models are valid in specific physical regimes described in terms of the parameters
\begin{eqnarray*}
\epsilon=\frac{a}{d_{1}},\; \mu=\frac{d_{1}}{\lambda^{2}},\; \epsilon_{2}=\frac{a}{d_{2}},\; \mu=\frac{d_{2}}{\lambda^{2}},
\end{eqnarray*}
where $a$ and $\lambda$ denote, respectively, a typical amplitude and wavelength of the interfacial deviation.

In the present paper we focus on two of these regimes. As previously mentioned, the upper layer is assumed to be shallow ($\mu<<1$) while for the lower layer the deformations are assumed to be of small amplitude ($\epsilon_{2}<<1$). Under these conditions, two situations are considered:
\begin{itemize}
\item[(i)] The Intermediate Long Wave (ILW) regime: In this case, small amplitude deformations are additionally assumed for the upper layer; specifically, it is supposed that
$$\mu\sim\epsilon^{2}\sim\epsilon_{2}<<1, \mu_{2}\sim 1.$$
In 1D, the corresponding system in nondimensional, unscaled form are given by the equations
\begin{equation}\label{ILW}
\begin{array}{l}
\left[1+\frac{\alpha}{\gamma}|D|{\rm coth}|D| \right]\zeta_t+\frac{1}{\gamma}\left((1- \zeta)u \right)_x+\frac{(\alpha-1)}{\gamma^2}(|D|{\rm coth}|D|)u_x=0\ ,\\
u_t+(1-\gamma)\zeta_x-\frac{1}{2\gamma}(u^2)_x=0\ ,
\end{array}
\end{equation}
where $\zeta=\zeta(x,t)$ denotes the interfacial deviation, $\gamma=\frac{\rho_{1}}{\rho_{2}}<1$,  $\alpha\geq 1$ is a modelling parameter, and the nonlocal operator $|D|$ has the Fourier symbol 
$$\widehat{|D|f}(k)=|k|\widehat{f}(k),\; k\in\mathbb{R},$$ with $\widehat{f}(k)$ standing for the Fourier transfom of $f$ at $k$. In (\ref{ILW}) $x$ and $t$ are proportional to distance along the fluid channel and time respectively, and $u=u(x,t)$ is a velocity variable.
\item[(ii)] The Benjamin-Ono (B-O) regime: This corresponds to the range of parameters
$$\mu\sim\epsilon^{2}\sim\epsilon_{2}<<1, \mu_{2}=\infty,$$ and the resulting 1D version of the corresponding systems in nondimensional, unscaled form is
\begin{equation}\label{BO}
\begin{array}{l}
\left[1+\frac{\alpha}{\gamma}|D| \right]\zeta_t+\frac{1}{\gamma}\left((1- \zeta)u \right)_x+\frac{(\alpha-1)}{\gamma^2}|D|u_x=0\ ,\\
u_t+(1-\gamma)\zeta_x-\frac{1}{\gamma}uu_{x}=0\ ,
\end{array}
\end{equation}
\end{itemize}
The linear well-posedness and consistency of the Euler system with (\ref{ILW}) and (\ref{BO}) are considered in  \cite{BLS2008}. The Cauchy problem for (\ref{ILW}) and (\ref{BO}) has been studied by Xu, \cite{X}. In the case of (\ref{ILW}), Xu showed, among other, local and long-time existence of solutions for $\alpha>1$. In addition, and for $\alpha>1$ as well, it is noted in \cite{X}, Remark 4.2, that these properties also hold for (\ref{BO}). Note that similar systems to (\ref{ILW}), (\ref{BO}) have been considered in \cite{CGK}.

Existence of smooth solitary-wave solutions of (\ref{ILW}) and (\ref{BO}) was recently proved in \cite{AnguloS2019}, with arguments based on the implicit function theorem. Furthermore, the solitary waves of the ILW systems were proved to decay exponentially, while those of the B-O systems to decay like $1/x^{2}$. The numerical generation of solitary waves of (\ref{ILW}) and (\ref{BO}) was studied in \cite{BonaDM2021}, where three iterative techniques, two of them based on the Petviashvili method, \cite{Petv1976,pelinovskys}, and the third given by the Conjugate-Gradient-Newton (CGN) method, \cite{Yang}, were introduced and their performance was compared. Approximations of some solitary-wave solutions for (\ref{ILW}) and (\ref{BO}), 
whose existence was not covered by the results in \cite{AnguloS2019}, were computed, and the resulting profiles were compared with solitary waves of the corresponding unidirectional ILW and BO equations.

In this paper we discretize in space the 
periodic initial-value problem (ivp) for the systems (\ref{ILW}) and (\ref{BO}), using the spectral Fourier-Galerkin method, and prove, in section \ref{sec2}, error estimates for the ensuing semidiscretizations. While there exist error analyses of spectral discretizations of the one-way ILW and B-O equations, cf. e.~g. \cite{PD}, we are not aware of any such analysis in the case of the systems.
Section \ref{sec3} is devoted to the numerical generation of solitary waves of the systems. Here we modify the Petviashvili iteration, implemented in \cite{BonaDM2021}, by introducing a vector extrapolation method, \cite{sidi}, with the aim of accelerating the convergence. Numerical examples illustrate the performance of the resulting procedure.

The following notation will be used. On the interval $(0,1)$,
the inner product and norm on $L^{2}=L^{2}(0,1)$ are denoted by $(\cdot,\cdot)$ and $||\cdot ||$, respectively. For real $\mu\geq 0$, $H^{\mu}$ will stand for the $L^{2}$-based periodic Sobolev spaces on $[0,1]$ with the norm given by
$$||g||_{\mu}=\left(\sum_{k\in\mathbb{Z}}(1+k^{2})^{\mu}|\widehat{g}(k)|^{2}\right)^{1/2},\; g\in H^{\mu},
$$ where $\widehat{g}(k)$ denotes the $k$th Fourier coefficient of $g$. For $1\leq p\leq\infty$ $W^{\mu,p}=W^{\mu,p}(0,1)$ stands for the Sobolev space of periodic functions on $(0,1)$ of order $\mu$, whose generalized derivatives are in $L^{p}=L^{p}(0,1)$. The norm on $L^{\infty}$
will be denoted by $|\cdot |_{\infty}$, and that on $W^{\mu,\infty}$ by $||\cdot ||_{\mu,\infty}$.
For an integer $N\geq 1$, $(\cdot,\cdot)_{N}$ will denote the Euclidean inner product in $\mathbb{C}^{2N}$, and the associated norm will be denoted by $||\cdot||_{N}$.
\section{Error estimates of the spectral semidiscretizations}
\label{sec2}
In order to describe and analyze the spectral semidiscretizations of (\ref{ILW}) and (\ref{BO}) some preliminaries are needed.
In the sequel we will assume that $\alpha>1$, so that the theory in \cite{X} is valid. We write the nonlocal operator in (\ref{ILW}) as
\begin{eqnarray}
g(D):=\frac{\alpha}{\gamma}|D|{\rm coth}|D|,\label{e21}
\end{eqnarray}
with symbol $g(k)=\frac{\alpha}{\gamma}|k|{\rm coth}|k|$. We observe that for real $k$ $g(k)$ behaves like $\frac{\alpha}{\gamma}|k|$ for $|k|>>1$, and like $\frac{\alpha}{\gamma}\left(1+\frac{k^{2}}{3}+O(k^{4})\right)$ for small $|k|$. Consequently, the symbol of $(1+g(D))^{-1}$, i.~e. $\left(1+\frac{\alpha}{\gamma}|k|{\rm coth}|k|\right)^{-1}$, is bounded for all $k\in\mathbb{R}$, and is of $O\left(\frac{1}{|k|}\right)$ as $|k|\rightarrow\infty$.

On the other hand, since in the B-O case the nonlocal term is 
\begin{eqnarray}
g(D)=\frac{\alpha}{\gamma}|D|,\label{e21b}
\end{eqnarray}
whose symbol is simply $\frac{\alpha}{\gamma}|k|$, we see again that the symbol of $(1+g(D))^{-1}$, i.~e. $\left(1+\frac{\alpha}{\gamma}|k|\right)^{-1}$, is bounded for all $k\in\mathbb{R}$ and is of $O\left(\frac{1}{|k|}\right)$ as $|k|\rightarrow\infty$.

We consider the periodic ivp for (\ref{ILW}) and (\ref{BO}) on the spatial interval $[0,1]$ and written for $0\leq x\leq 1,\; 0\leq t\leq T$ in the form
\begin{eqnarray}
&&\zeta_{t}+\frac{1}{\gamma}(1+g(D))^{-1}\left(1+\frac{(\alpha-1)}{\alpha}g(D)\right)u_{x}=\frac{1}{\gamma}(1+g(D))^{-1}(\zeta u)_{x},\nonumber\\
&&u_{t}+(1-\gamma)\zeta_{x}=\frac{1}{2\gamma}\partial_{x}(u^{2}),\label{e22}\\
&&\zeta(x,0)=\zeta_{0}(x),\; u(x,0)=u_{0}(x),\nonumber
\end{eqnarray}
where $\zeta_{0}, u_{0}$ are given $1$-periodic smooth functions on $[0,1]$. We note again that for the ILW system (\ref{ILW}) $g(D)$ is given by (\ref{e21}), while, for the B-O system, by (\ref{e21b}).
We assume that the ivp (\ref{e22}) has a unique solution which is sufficiently smooth for the purposes of the error estimation.

We introduce the nonlocal operators $$\mathcal{T}:=(1+g(D))^{-1},\quad \mathcal{J}:=(1+g(D))^{-1}\left(1+\frac{(\alpha-1)}{\alpha}g(D)\right).$$ We have previously examined the symbol of $\mathcal{T}$. The symbol of $\mathcal{J}$ is $\frac{1+\frac{(\alpha-1)}{\alpha}g(k)}{1+g(k)}$, which, in view of our assumptions on $\alpha$ and $\gamma$, is well defined and bounded for $k\in\mathbb{R}$.

Let $N\geq 1$ be an integer, and consider the finite dimensional space
\begin{eqnarray*}
S_{N}={\rm span} \{e^{ikx},\; k\in\mathbb{Z}, -N\leq k\leq N\}.
\end{eqnarray*}
We recall several properties of 
the $L^{2}$-projection operator onto $S_{N}$
\begin{eqnarray*}
P_{N}v=\sum_{|k|\leq N}\widehat{v}(k)e^{ikx},
\end{eqnarray*}
where $\widehat{v}(k)$ is the $k$th Fourier coefficient of $v$. 
\begin{itemize}
\item $P_{N}$ commutes with $\partial_{x}$. 
\item Given integers $0\leq j\leq \mu$, and for any $v\in H^{\mu}, \mu\geq 1$,
\begin{eqnarray}
||v-P_{N}v||_{j}&\leq &CN^{j-\mu}||v||_{\mu},\label{epn1}\\
|v-P_{N}v|_{\infty}&\leq &CN^{1/2-\mu}||v||_{\mu},\label{epn2}
\end{eqnarray}
for some constant $C$ independent of $N$. 
\end{itemize}
We will also use the following inverse inequalities.
Given $0\leq j\leq \mu$, there exists a constant $C$ independent of $N$, such that for any $\psi\in S_{N}$
\begin{eqnarray}
||\psi||_{\mu}\leq CN^{\mu-j}||\psi||_{j},\; 
||\psi||_{\mu,\infty}\leq CN^{1/2+\mu-j}||\psi||_{j}.\label{epn3}
\end{eqnarray}

In what follows, $C$ will denote a constant independent of $N$.

Now we may define the semidiscretizations.
The semidiscrete spectral Fourier-Galerkin approximation of (\ref{e22}) is defined by the real-valued functions $\zeta_{N},u_{N}:[0,T]\rightarrow S_{N}$ satisfying for $0\leq t\leq T$
\begin{eqnarray}
&&\zeta_{N,t}+\frac{1}{\gamma}(1+g(D))^{-1}\left(1+\frac{(\alpha-1)}{\alpha}g(D)\right)u_{N,x}=\frac{1}{\gamma}(1+g(D))^{-1}\partial_{x}P_{N}(\zeta_{N} u_{N}),\nonumber\\
&&u_{N,t}+(1-\gamma)\zeta_{N,x}=\frac{1}{2\gamma}\partial_{x}P_{N}(u_{N}^{2}), \label{e23}\\
&&\zeta_{N}\big|_{t=0}=P_{N}\zeta_{0},\; u_{N}\big|_{t=0}=P_{N}u_{0}.\nonumber
\end{eqnarray}
The ode ivp (\ref{e23}) is implemented in its Fourier component form 
\begin{eqnarray*}
&&\widehat{\zeta}_{N,t}+\frac{1}{\gamma}(1+g(k))^{-1}\left(1+\frac{(\alpha-1)}{\alpha}g(k)\right)(ik)\widehat{u}_{N}=\frac{1}{\gamma}(1+g(k))^{-1})ik)\widehat{\zeta_{N} u_{N}},\\
&&\widehat{u}_{N,t}+(1-\gamma)(ik)\widehat{\zeta}_{N}=\frac{1}{2\gamma}(ik)\widehat{u_{N}^{2}},\\
&&\widehat{\zeta}_{N}(k,0)=\widehat{\zeta}_{0}(k),\; \widehat{u}_{N}(k,0)=\widehat{u}_{0}(k),
\end{eqnarray*}
where $\widehat{\zeta}_{N}=\widehat{\zeta}_{N}(k,t), \widehat{u}_{N}=\widehat{u}_{N}(k,t), -N\leq k\leq N, t\geq 0$ denote the Fourier coefficients of $\zeta_{N}$ and $u_{N}$ respectively.

The ode ivp (\ref{e23}) clearly has a local in time solution; part of the proof of the following proposition is showing that this solution can be extended up to $t=T$.
\begin{proposition}
\label{pro21}
(ILW systems.) Assume that the solution $\zeta, u$ of (\ref{e22}) for $\alpha>1$, and $g$ given by (\ref{e21}),  is such that $\zeta, u\in H^{\mu}, \mu>3/2$ for $0\leq t\leq T$. Then, for $N$ sufficiently large,
\begin{eqnarray}
\max_{0\leq t\leq T}\left(||\zeta_{N}-\zeta||+||u_{N}-u||\right)\leq CN^{1-\mu}.\label{e24}
\end{eqnarray}
\end{proposition}
\begin{proof}
The general plan of the proof resembles that of \cite{DDS1}, where another class of asymptotic models for internal waves was analyzed.
We let $\theta=\zeta_{N}-P_{N}\zeta,\rho=P_{N}\zeta-\zeta$, so that $\zeta_{N}-\zeta=\theta+\rho$, and $\xi=u_{N}-P_{N}u, \sigma=P_{N}u-u$, so that $u_{N}-u=\xi+\sigma$. Applying $P_{N}$ on both sides of the pde's in (\ref{e22}) and subtracting from the respective semidiscrete equations in (\ref{e23}) we obtain
\begin{eqnarray}
&&\theta_{t}+\frac{1}{\gamma}\mathcal{J}\xi_{x}=\frac{1}{\gamma}\mathcal{T}\partial_{x}P_{N}A,\nonumber\\
&&\xi_{t}+(1-\gamma)\theta_{x}=\frac{1}{2\gamma}\partial_{x}P_{N}B,\qquad t\geq 0\label{e25}\\
&&\theta\big|_{t=0}=0,\; \xi\big|_{t=0}=0,\nonumber
\end{eqnarray}
where it is straightforward to see that
\begin{eqnarray*}
A&=&u\rho+\zeta \sigma+u\theta+\zeta\xi+\sigma\theta+\rho\xi+\rho\sigma+\theta\xi,\\
B&=&u\sigma+u\xi+\sigma\xi+\frac{1}{2}\sigma^{2}+\frac{1}{2}\xi^{2}.
\end{eqnarray*}
Using a standard trick for rational functions, cf. e.~g. \cite{X}, we may write
$$\mathcal{J}=(1+g(D))^{-1}\left(1+\frac{(\alpha-1)}{\alpha}g(D)\right)=\frac{\alpha-1}{\alpha}+\frac{1}{\alpha}(1+g(D))^{-1}=\frac{\alpha-1}{\alpha}+\frac{1}{\alpha}\mathcal{T},$$ and, therefore, may simplify the equations somewhat having just one nonlocal operator in the problem. We rewrite accordingly (\ref{e25}) as
\begin{eqnarray}
&&\theta_{t}+\frac{1}{\gamma}\frac{\alpha-1}{\alpha}\xi_{x}+\frac{1}{\alpha\gamma}\mathcal{T}\xi_{x}=\frac{1}{\gamma}\mathcal{T}\partial_{x}P_{N}A,\nonumber\\
&&\xi_{t}+(1-\gamma)\theta_{x}=\frac{1}{2\gamma}\partial_{x}P_{N}B,\qquad t\geq 0,\label{e26}\\
&&\theta\big|_{t=0}=0,\; \xi\big|_{t=0}=0.\nonumber
\end{eqnarray}
We will use the standard energy method to estimate $\theta$ and $\xi$. Taking the $L^{2}$ inner products of the semidiscrete equations in (\ref{e26}) with $\theta$ and $\xi$, respectively, we have
\begin{eqnarray}
&&(\theta_{t},\theta)+\frac{1}{\gamma}\frac{\alpha-1}{\alpha}(\xi_{x},\theta)+\frac{1}{\alpha\gamma}(\mathcal{T}\xi_{x},\theta)=\frac{1}{\gamma}(\mathcal{T}A_{x},\theta),\label{e27}\\
&&(\xi_{t},\xi)-(1-\gamma)(\theta,\xi_{x})=\frac{1}{2\gamma}(B_{x},\xi).\label{e28}
\end{eqnarray}
Hence, multiplying (\ref{e28}) by $\frac{\alpha-1}{\alpha\gamma (1-\gamma)}$ and adding to (\ref{e27}) we obtain
\begin{eqnarray}
\frac{1}{2}\frac{d}{dt}\left(||\theta||^{2}+\frac{\alpha-1}{\alpha\gamma (1-\gamma)}||\xi||^{2}\right)&=&-\frac{1}{\alpha\gamma}(\mathcal{T}\xi_{x},\theta)+\frac{1}{\gamma}(\mathcal{T}A_{x},\theta)\nonumber\\
&&+\frac{\alpha-1}{2\alpha\gamma^{2}(1-\gamma)}(B_{x},\xi).\label{e29}
\end{eqnarray}
Note that $\frac{\alpha-1}{2\alpha\gamma^{2}(1-\gamma)}>0$. Also note that the operator $\mathcal{T}\partial_{x}$, with symbol $\frac{ik}{1+g(k)}$, is bounded in $L^{2}$. Therefore we have by (\ref{e29}), as long as the solution of (\ref{e26}) (or (\ref{e23})) exists, that
\begin{eqnarray}
\frac{d}{dt}\left(||\theta||^{2}+||\xi||^{2}\right)\leq C\left(||\xi|| ||\theta||+||A|| ||\theta|| +|(B_{x},\xi)|\right).\label{e210}
\end{eqnarray}
We bound now the two last terms in the right-hand side of (\ref{e210}). By the definition of $A$ we have
\begin{eqnarray*}
||A||&\leq & |u|_{\infty}||\rho||+|\zeta|_{\infty}||\sigma||+|u|_{\infty}||\theta||+|\zeta|_{\infty}||\xi||+|\sigma|_{\infty}||\theta||\\
&&+|\rho|_{\infty}||\xi||+|\rho|_{\infty}||\sigma||+|\theta|_{\infty}||\xi||.
\end{eqnarray*}
Let $t_{N}\in (0,T]$ be the maximal temporal instance such that
\begin{eqnarray}
|\theta|_{\infty}\leq 1,\; 0\leq t\leq t_{N}.\label{e211}
\end{eqnarray}
Then, by the approximation properties of $S_{N}$ (\ref{epn1}), (\ref{epn2}), the inverse inequalities (\ref{epn3}), and the fact that $\mu\geq 1$, we get by the above estimate of $||A||$ that for $0\leq t\leq t_{N}$
\begin{eqnarray}
||A||\leq C\left(N^{-\mu}+||\theta||+||\xi||\right),\label{e212}
\end{eqnarray}
where $C$ is independent of $N$. To bound the term $(B_{x},\xi)$ note that by periodicity
\begin{eqnarray}
(B_{x},\xi)=((u\sigma)_{x},\xi)+((u\xi)_{x},\xi)+((\sigma\xi)_{x},\xi)+(\sigma\sigma_{x},\xi).\label{e213}
\end{eqnarray}
Since $\mu\geq 3/2$
\begin{eqnarray*}
|((u\sigma)_{x},\xi)|\leq |u_{x}|_{\infty}||\sigma|| ||\xi||+|u|_{\infty}||\sigma_{x}|| ||\xi||\leq CN^{1-\mu}||\xi||.
\end{eqnarray*}
Also, for the same reason
\begin{eqnarray*}
|((u\xi)_{x},\xi)|&= &\frac{1}{2}|(u_{x}\xi,\xi)|\leq \frac{1}{2}|u_{x}|_{\infty}||\xi||^{2}\leq C||\xi||^{2},\\
|((\sigma\xi)_{x},\xi)|&= &\frac{1}{2}|(\sigma_{x}\xi,\xi)|\leq \frac{1}{2}|\sigma_{x}|_{\infty}||\xi||^{2}\leq C||\xi||^{2},
\end{eqnarray*}
and
\begin{eqnarray*}
|(\sigma\sigma_{x},\xi)|\leq |\sigma_{x}|_{\infty}||\sigma|| ||\xi||\leq CN^{\frac{3}{2}-2\mu}||\xi||\leq CN^{-\mu}||\xi||.
\end{eqnarray*}
From these estimates and (\ref{e213}) we conclude that
\begin{eqnarray}
|(B_{x},\xi)|\leq C\left(N^{2(1-\mu)}+||\xi||^{2}\right),\label{e214}
\end{eqnarray}
as long as the solution of (\ref{e26}) exists. Therefore, (\ref{e210}), (\ref{e212}) and (\ref{e214}) give for $0\leq t\leq t_{N}$
\begin{eqnarray*}
\frac{d}{dt}(||\theta||^{2}+||\xi||^{2})\leq C\left(N^{2(1-\mu)}+||\theta||^{2}+||\xi||^{2}\right),
\end{eqnarray*}
from which, by Gronwall's inequality we get
\begin{eqnarray}
||\theta||+||\xi||\leq CN^{1-\mu},\label{e215}
\end{eqnarray}
for $0\leq t\leq t_{N}$, where $C$ is independent of $N$ and $t_{N}$. Since by (\ref{e215}) $|\theta|_{\infty}\leq CN^{3/2-\mu}$ and $\mu>3/2$, we infer that $t_{N}$ was not maximal in (\ref{e211}) for $N$ sufficiently large, and in the customary way the existence of solutions of (\ref{e25}) and the validity of (\ref{e215}) may be extended to $t=T$. The estimate (\ref{e24}) follows.
\end{proof}

As far as the B-O case is concerned, recall that the symbol of $(1+g(D))^{-1}$ is also bounded for all $k\in\mathbb{R}$ and is of $O\left(\frac{1}{|k|}\right)$ as $|k|\rightarrow\infty$. This was the basic property that we used in the error analysis of the spectral semidiscretization of (\ref{ILW}). Hence the proof of Proposition \ref{pro21} can be easily adapted to the B-O case for $\alpha>1$. Without proof we state:

\begin{proposition}
\label{pro31}
(B-O systems.) Assume that the solution $\zeta, u$ of the periodic ivp for (\ref{BO}) for $\alpha>1$ is such that $\zeta, u\in H^{\mu}, \mu>3/2$ for $0\leq t\leq T$. Let $(\zeta_{N},u_{N})$ be the solution of the Fourier-Galerkin semidiscretization of the periodic ivp, defined by (\ref{e23}), where now $g$ is given by (\ref{e21b}). Then $(\zeta_{N},u_{N})$ exists uniquely up to $t=T$ and satisfy
for $N$ sufficiently large
\begin{eqnarray*}
\max_{0\leq t\leq T}\left(||\zeta_{N}-\zeta||+||u_{N}-u||\right)\leq CN^{1-\mu},
\end{eqnarray*}
for some constant $C$ independent of $N$.
\end{proposition}
\section{Solitary wave solutions}
\label{sec3}
The ILW and B-O systems (\ref{ILW}) and (\ref{BO}) have been shown to posses solitary-wave solutions. These are solutions $\zeta=\zeta(x-ct), u=u(x-ct), c\neq 0$, where $\zeta(X), u(X)\rightarrow 0$ as $|X|\rightarrow\infty$, and that satisfy the system
\begin{eqnarray}
&&-c(1+g(D))\zeta+\frac{1}{\gamma}\left(1+\frac{(\alpha-1)}{\alpha}g(D)\right)u=\frac{1}{\gamma}\zeta u,\nonumber\\
&&-cu+(1-\gamma)\zeta=\frac{1}{2\gamma}u^{2},\label{e31}
\end{eqnarray}
where $g(D)$ is given by (\ref{e21}) or (\ref{e21b}). The existence of smooth solutions of (\ref{e31}) was proved by Angulo-Pava and Saut, \cite{AnguloS2019}, for some range of speeds $c$, using the implicit function theorem. Properties of their asymptotic decay were also proved in the same paper, ensuring that in the ILW case the solitary waves decay exponentially, while in the B-O case, the decay is algebraic, like $1/|X|^{2}$. (Note that, in both cases, the asymptotic behaviour is the same as that of the solitary wave solutions of the corresponding unidirectional models.)

The numerical generation of approximate solutions of (\ref{e31}) was studied in \cite{BonaDM2021}. We summarize the numerical technique used. Let $l>0$ be large enough, $N\geq 1$ be an even integer, and discretize the periodic problem for (\ref{e31}) on $[-l,l]$ with a Fourier collocation method based on the $N$ collocation points $x_{j}=-l+jh, j=0,\ldots,N-1, h=2l/N$. The approximation to the solitary wave $(\zeta,u)$ is then represented by the nodal values $\zeta_{h}=(\zeta_{h,0},\ldots,\zeta_{h,N-1})^{T}$ and $u_{h}=(u_{h,0},\ldots,u_{h,N-1})^{T}$, where $\zeta_{h,j}\approx \zeta(x_{j}), u_{h,j}\approx u(x_{j}), j=0,\ldots, N-1$, and $\zeta_{h}, u_{h}$ satisfy the system
\begin{eqnarray}
S\begin{pmatrix} \zeta_{h}\\u_{h}\end{pmatrix}=F(\zeta_{h},u_{h}):=\frac{1}{\gamma}\begin{pmatrix} \zeta_{h}.u_{h}\\(u_{h}.^{2})/2\end{pmatrix},\label{e32}
\end{eqnarray}
where $S$ is the $2N$-by-$2N$ matrix
\begin{eqnarray}
S:=\begin{pmatrix} -c(I_{N}+g(D_{N}))&\frac{1}{\gamma}(I_{N}+\frac{\alpha-1}{\alpha}g(D_{N}))\\(1-\gamma)I_{N}& -cI_{N}\end{pmatrix},\label{e33}
\end{eqnarray}
being $I_{N}$ the $N$-by-$N$ identity matrix and $D_{N}$  the $N$-by-$N$ Fourier pseudospectral differentiation matrix. The dots on the right-hand side of (\ref{e32}) signify Hadamard products. The system (\ref{e32}), (\ref{e33}) is implemented in its Fourier component form. Thus for $-N/2\leq k\leq N/2-1$, the $k$th discrete Fourier components of $\zeta_{h}$ and $u_{h}$, denoted by $\widehat{\zeta_{h}}(k), \widehat{v_{h}}(k) $, resp., satisfy the fixed point system
\begin{eqnarray}
\underbrace{\begin{pmatrix}-c(1+g(\widetilde{k}))&\frac{1}{\gamma}(1+\frac{\alpha-1}{\alpha}g(\widetilde{k}))\\(1-\gamma)& -c\end{pmatrix}}_{S(\widetilde{k})}
\begin{pmatrix} \widehat{\zeta_{h}}(\widetilde{k})\\\widehat{v_{h}}(\widetilde{k})\end{pmatrix}=\underbrace{\frac{1}{\gamma}\begin{pmatrix} \widehat{\zeta_{h}.u_{h}}(\widetilde{k})\\\widehat{(u_{h}.^{2})/2}(\widetilde{k})\end{pmatrix}}_{\widehat{F(\zeta_{h},u_{h})}_{\widetilde{k}}},\label{e34}
\end{eqnarray}
where
$\widetilde{k}=\pi k/l, -N/2\leq k\leq N/2-1$.

In \cite{BonaDM2021} three methods for the iterative resolution of (\ref{e34}) were proposed: The Petviashvili iteration, the CGN method, and a variant of Petviashvili's method (called e-Petviashvili's method) obtained from solving $\zeta$ in terms of $u$ in the second equation of (\ref{e31}) and substituting into the first one. The resulting equation for $u$ (with quadratic and cubic terms) is iteratively solved with the method proposed in \cite{AlvarezD2014}, an extension of the Petviashvili scheme for nolinearities which are superpositions of homogeneous functions with different degree of homogeneity, cf. \cite{BonaDM2021} for details.
\subsection{Numerical generation of solitary waves with acceleration methods}
In the present paper we propose an alternative technique based on implementing the Petviashvili iteration combined with a vector extrapolation method, \cite{sidi}. The inclusion of the extrapolation has the general benefit of accelerating the convergence of the basic method used for the iteration. (In some cases, the process changes from divergent to convergent.) 

In the present case, the Petviahsvili iteration, formulated as, cf. \cite{BonaDM2021},
\begin{eqnarray}
m_{\nu}&=&\frac{\langle S_{N}Z^{[\nu]},Z^{[\nu]}\rangle_{N}}{\langle F(Z^{[\nu]}),Z^{[\nu]} \rangle_{N}},\nonumber\\
SZ^{[\nu+1]}&=&m_{\nu}^{2}F(Z^{[\nu]}),\; \nu=0,1,\ldots,\label{e35}
\end{eqnarray}
where $Z^{[\nu]}:=(\zeta_{h}^{[\nu]},u_{h}^{\nu]})$, is combined with the so-called minimal polynomial extrapolation method (MPE), which may be described as follows (cf. \cite{smithfs,AlvarezD2015} and references therein for details). From a number of $l$ iterations $Z^{[\nu]},\ldots,Z^{[\nu+l]}$ with the Petviashvili method (\ref{e35}), the extrapolation steps
\begin{eqnarray}
X_{\nu,l}=\sum_{j=0}^{l}\gamma_{j}Z^{[\nu+j]},\label{e34b}
\end{eqnarray}
are computed, where the coefficients $\gamma_{j}$ are of the form
\begin{eqnarray}
\gamma_{j}=\frac{c_{j}}{\displaystyle\sum_{i=0}^{l}c_{i}},\quad 0\leq j\leq l,\label{e34c}
\end{eqnarray}
with $c_{l}=1$ and the $c_{j}, j=0,\ldots,l-1$, are the solution, in the least squares sense, of the system
\begin{eqnarray*}
\sum_{i=0}^{l-1}c_{i}W_{\nu+i}=\widetilde{W}_{\nu},
\end{eqnarray*}
where $W_{j}=\Delta Z^{[j]}:=Z^{[j+1]}-Z^{[j]}, \widetilde{W}_{j}=-\Delta Z^{[j+l]}$. The method (\ref{e34b}), (\ref{e34c}) was originally formulated and analyzed for linear vector sequences in \cite{CabayJ1976}.

The MPE method is typically implemented in cycling mode, \cite{smithfs}. For a fixed width of extrapolation $mw\geq 1$, the advance from the $(\nu+1)$th iteration is performed according to the following steps:
\begin{itemize}
\item Step 1: Compute $mw$ steps of (\ref{e35}) from $X^{[0]}=Z^{[\nu]}$: $Z^{[1]},\ldots,Z^{[mw]}$.
\item Step 2: Compute the corresponding extrapolation steps (\ref{e34b}) from the iterations of Step 1.
\item Set $Z^{[\nu+1]}=X_{mw,\nu}$, $X^{[0]}=Z^{[\nu+1]}$ and go to Step 1.
\end{itemize}

The resulting procedure is controlled by iterating while the residual error
\begin{eqnarray}
RES(\nu)=||S_{N}Z^{[\nu]}-F(Z^{[\nu]})||_{N},\label{e36}
\end{eqnarray}
 is above some tolerance.

\subsection{Numerical experiments}
We illustrate the iterative method described in the previous section with some numerical experiments.
\begin{figure}[htbp]
\centering
\centering
\subfigure[]
{\includegraphics[width=10cm,height=5.2cm]{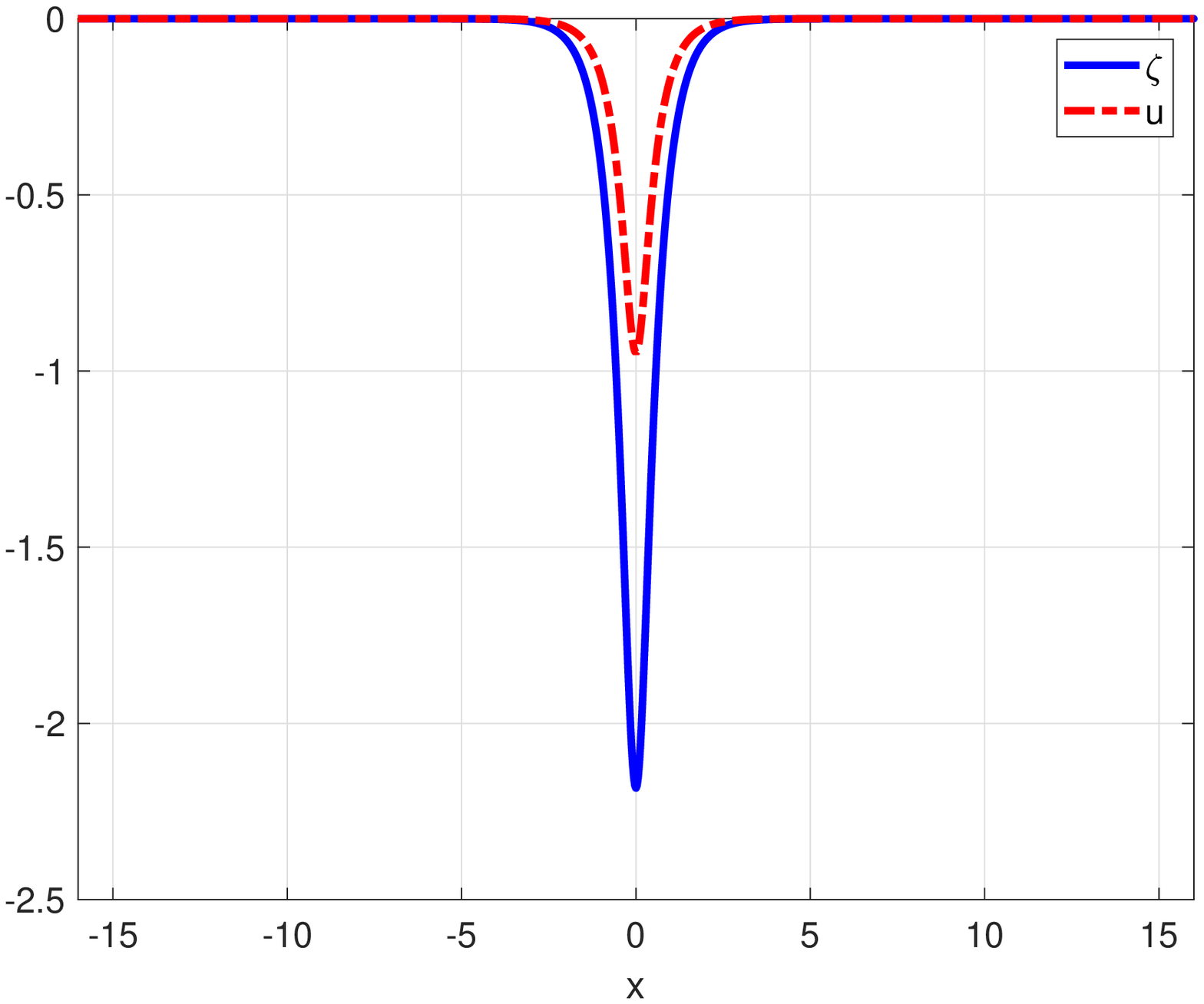}}
\subfigure[]
{\includegraphics[width=10cm,height=5.2cm]{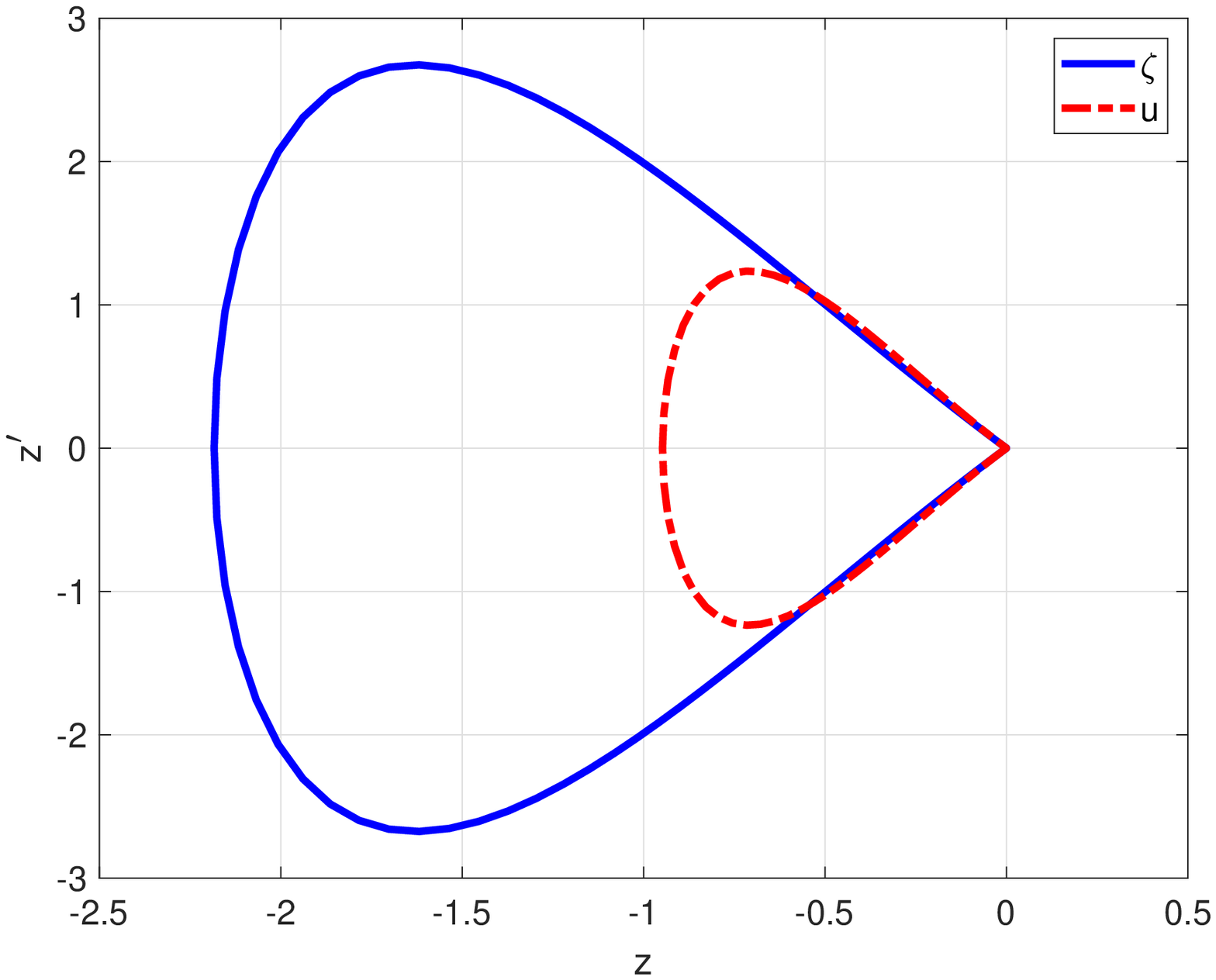}}
\subfigure[]
{\includegraphics[width=10cm,height=5.2cm]{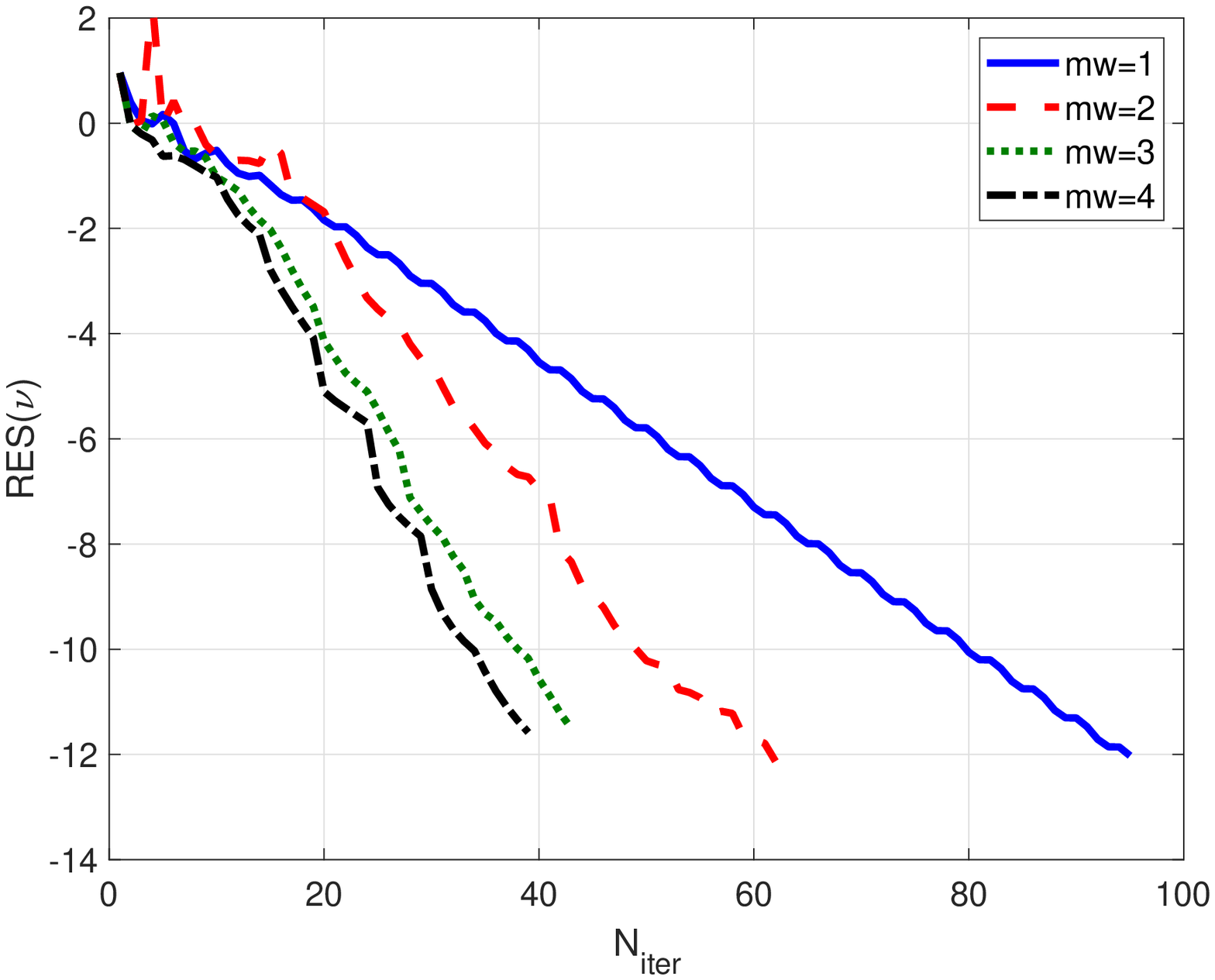}}
\caption{Numerical generation of solitary waves. ILW case with $\gamma=0.8, \alpha=1.2, c_{s}=0.52$. (a) $\zeta$ and $u$ numerical profiles; (b) Phase plot; (c) Residual error (\ref{e36}) vs. number of iterations for several values of the width of extrapolation $mw$.}
\label{ILW1}
\end{figure}
In the case of the ILW system (\ref{ILW}), Figure \ref{ILW1}(a) shows the numerical solitary wave profile obtained with the Petviashvili method (without extrapolation) for $\gamma=0.8, \alpha=1.2$, $c=0.52$ , while the corresponding phase plots are shown in Figure \ref{ILW1}(b). The effect of the extrapolation technique is observed in Figure \ref{ILW1}(c), which shows the behaviour of the residual error (\ref{e36}) as function of the number of iterations, for several values of the width $mw$ ($mw=1$ would correspond to the iteration without extrapolation). Note that the inclusion of the vector extrapolation accelerates the convergence of the iteration by diminishing the number of iterations required for the residual to become smaller than a fixed error.
\begin{figure}[htbp]
\centering
\centering
\subfigure[]
{\includegraphics[width=10cm,height=5.2cm]{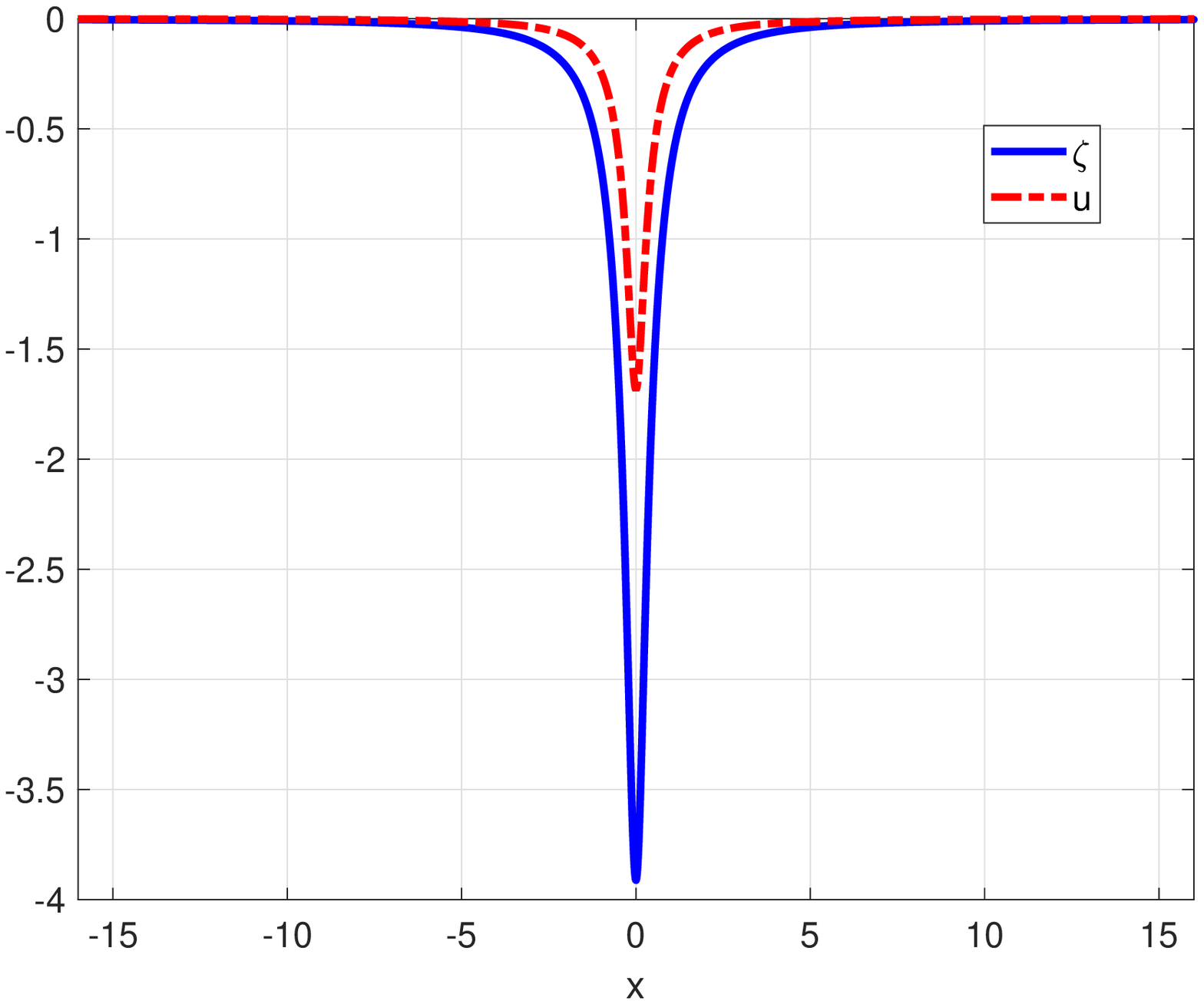}}
\subfigure[]
{\includegraphics[width=10cm,height=5.2cm]{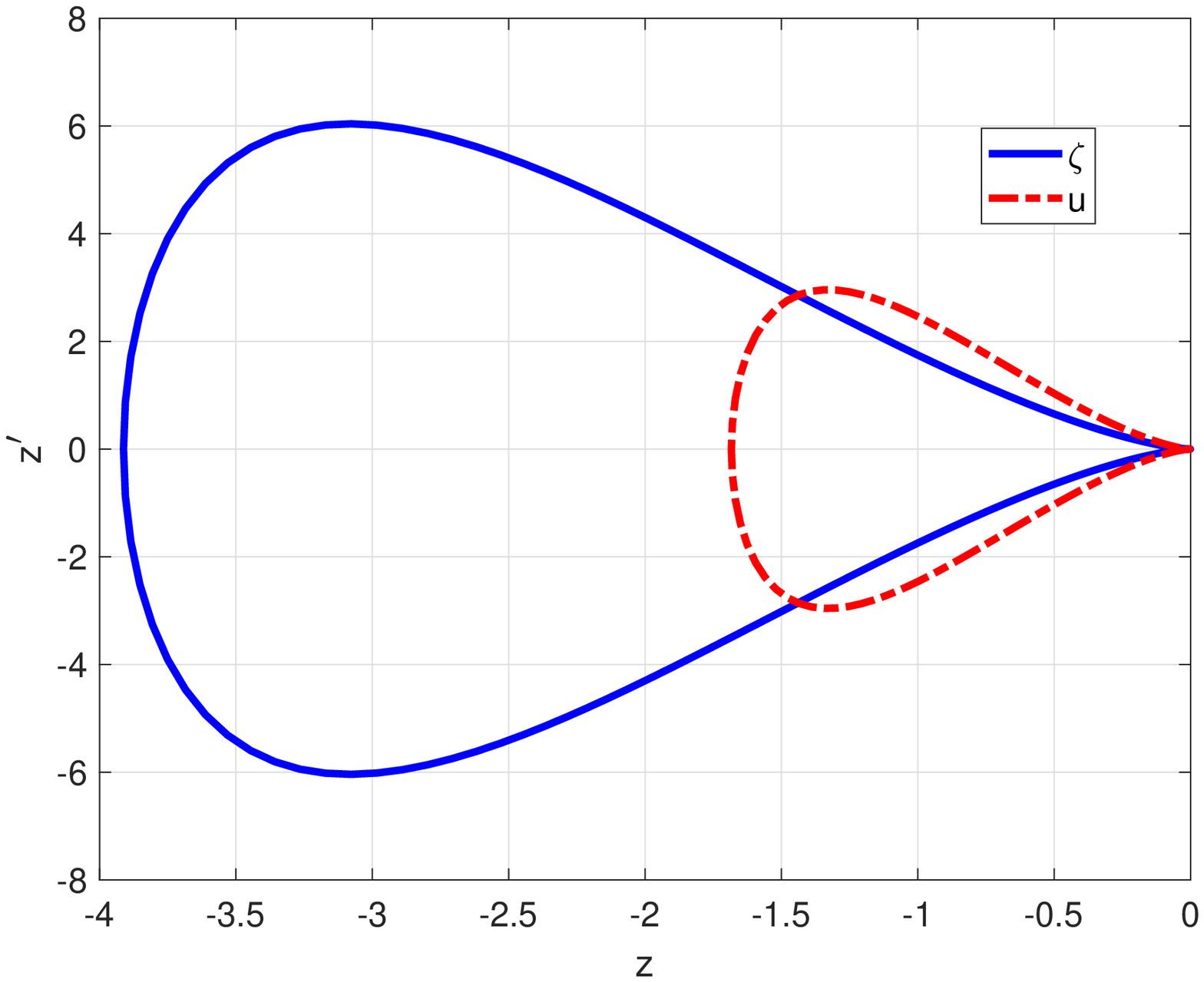}}
\subfigure[]
{\includegraphics[width=10cm,height=5.2cm]{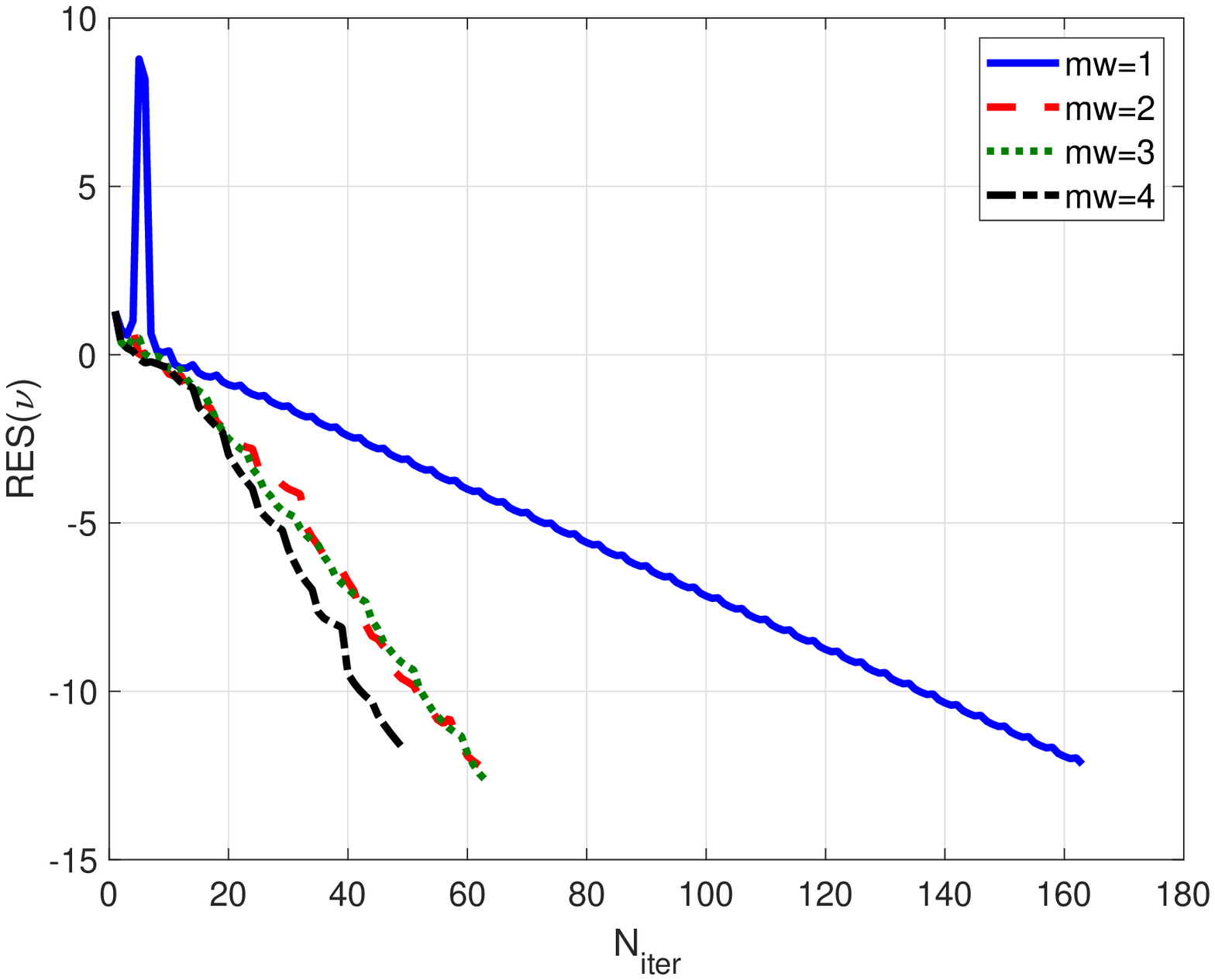}}
\caption{Numerical generation of solitary waves. B-O case with $\gamma=0.8, \alpha=1.2, c_{s}=0.57$. (a) $\zeta$ and $u$ numerical profiles; (b) Phase plot; (c) Residual error (\ref{e36}) vs. number of iterations for several values of the width of extrapolation $mw$.}
\label{BO1}
\end{figure}
The illustration of the B-O case is given in Figure \ref{BO1}, where an approximate solitary wave solution of (\ref{e31}), in the B-O case, for $\gamma=0.8, \alpha=1.2$, and $c=0.57$ is shown in Figure \ref{BO1}(a), and with the corresponding phase plot in Figure \ref{BO1}(b). By comparison with Figure \ref{ILW1}(b), we can observe the different type of decay to zero at infinity (algebraic versus exponential, cf. \cite{AnguloS2019}) of the solitary waves as trajectories homoclinic to the origin. The acceleration of the convergence for the BO case is shown in Figure \ref{BO1}(c). The most remarkable reduction in the number of iterations for a given residual error is observed when we let $mw=1$ (no acceleration) and then $mw=2$ (cycling mode extrapolation with two steps of the Petviashvili method). After that, for larger values of $mw$, the method continues accelerating the convergence, with a milder reduction in the number of iterations (see e.~g. \cite{smithfs,AlvarezD2015} for discussions about an optimal choice for $mw$).
\section*{Acknowledgements}
Vassilios Dougalis and Angel Dur\'an would like to acknowledge travel support, that made possible this collaboration, from the Institute of Mathematics (IMUVA) of the University of Valladolid, and the Institute of Applied and Computational Mathematics of FORTH. Angel Dur\'an was supported by Junta de Castilla y Le\'on and FEDER funds (EU) under Research Grant VA193P20.
Leetha Saridaki was supported by the grant \lq\lq Innovative Actions in Environmental Research and Development (PErAn)\rq\rq (MIS5002358), implemented under the \lq\lq Action for the strategic development of the Research and Technological sector'' funded by the Operational Program \lq\lq Competitiveness, and Innovation'' (NSRF 2014-2020) and co- financed by Greece and the EU (European Regional Development Fund). The grant was issued to the Institute of Applied and Computational Mathematics of FORTH.

\end{document}